\documentclass{amsart}

\usepackage{amssymb}
\usepackage{graphicx}

\newtheorem{theorem}{Theorem}[section]
\newtheorem{lemma}[theorem]{Lemma}
\newtheorem{corollary}[theorem]{Corollary}
\newtheorem{remark}[theorem]{Remark}

\theoremstyle{definition}
\newtheorem{definition}[theorem]{Definition}

\theoremstyle{remark}

\newtheorem{problem}[theorem]{Problem}

\numberwithin{equation}{section}

\begin{document}

\title[Restriction of averaging operators]{Restriction of averaging operators to algebraic varieties over finite fields}

\author{Doowon Koh* }
\address{Department of Mathematics\\
Chungbuk National University \\
Cheongju city, Chungbuk-Do 28644 Korea}
\email{koh131@chungbuk.ac.kr}
\thanks{*Corresponding Author\\ This research was supported by  the research grant of   Basic Science Research Program through the National Research Foundation of Korea  funded by the Ministry of Education, Science and Technology (NRF-2015R1A1A1A05001374).}

\author{Seongjun Yeom}
\address{Department of Mathematics\\
Chungbuk National University \\
Cheongju city, Chungbuk-Do 28644 Korea}

\email{mathsj@chungbuk.ac.kr}
\thanks{}

\subjclass[2000]{Primary: 42B05 ; Secondary 11T23 }



\keywords{Algebraic curves, finite fields, restricted averaging operators}

\begin{abstract} We study $L^p\to L^r$ estimates for restricted averaging operators related to algebraic varieties $V$ of $d$-dimensional vector spaces over finite fields $\mathbb F_q$ with $q$ elements.
We observe properties of both the Fourier restriction operator and the averaging operator over $V\subset \mathbb F_q^d.$ As a consequence, we obtain optimal results on the restricted averaging problems for spheres and  paraboloids in dimensions $d\ge2,$ and  cones in odd dimensions $d\ge 3.$  In addition, when  the variety $V$ is a cone lying in an even dimensional vector space over $\mathbb F_q$ and  $-1$ is a square number in $\mathbb F_q$,  we also obtain sharp estimates except for two endpoints.

\end{abstract}

\maketitle

\section{Introduction}
Over the past decade there has been a lot of interest in developing harmonic analysis over finite fields.  
Mockenhaupt and Tao in \cite{MT04} initially studied the finite field restriction problem. The finite field Kakeya problem was posed by Wolff in \cite{Wo}.  These two problems are considered as central problems in Euclidean harmonic analysis. Amazingly, Dvir  in \cite{Dv} recently found a simple proof of the finite field Kakeya conjecture, wherein he invoked the polynomial method. 
There are some serious difficulities in adapting the Dvir's proof to the Euclidean Kakeya problem.
On the other hand the polynomial method plays an important role in improving some problems in harmonic analysis. For instance, Guth in \cite{Gu15} used it to obtain an improvment on the Euclidean  restriction problem (see also \cite{Gu, GK}). This example demonstrates that  finite field analogues can be useful in developing methods for Euclidean analogs.   \\

Finite fields can be an efficient method by which one can introduce a problem to mathematicians in other fields. This is mainly due to finite fields possessing a relatively simple structure.  Furthermore, problems in finite fields are closely related to other mathematical subjects such as algebraic geometry, additive number theory, or combinatorics. 
For these reasons, analysis problems in finite fields  have received much attention in the last few decades   (see, for example, \cite{Ca06, EOT, IK10, Le13, Le14} ). \\

In this paper we study a hybrid of the averaging operators and restriction operators in the finite field setting. Roughly speaking, given an algebraic variety $V\subset \mathbb F_q^d$ and  an  operator $T$ acting on functions $f: \mathbb F_q^d\to \mathbb C,$ a new operator $T_V$ can be defined by restricting $Tf$ to the variety $V$. Then a natural question is to determine the boundedness of the restricted operator $T_V.$  Note that when  $Tf$ is the Fourier transform of $f$, this problem becomes the Fourier restriction problem.   
While it is possible to study the extended Fourier restriction problem for various operators, we shall focus on studying the problem for averaging operators $T$ over algebraic varieties $V\subset \mathbb F_q^d.$ We call this the restricted averaging problem to $V.$ It was observed in \cite{Ko15} that  optimal results can be obtained if the variety $V$ is any curve on two dimensions which does not contain a line. In this paper we extend the work to higher dimensions.

\subsection{Discrete Fourier analysis}
We begin by reviewing definitions and  notation. Let $\mathbb F_q^d$ be a $d$-dimensional vector space over the finite field $\mathbb F_q$ with $q$ elements. Throughout this paper we always assume that $q$ is the power of an odd prime. We endow $\mathbb F_q^d$ with the normalized counting measure $dx.$ We shall use $(\mathbb F_q^d, dx)$ to indicate the $d$-dimensional vector space with the normalized counting measure $dx$. In order to indicate the dual space of $(\mathbb F_q^d, dx)$, we shall use $(\mathbb F_q^d, dm)$  which is endowed with the counting measure $dm.$  Thus,
$$ \int_{\mathbb F_q^d} f(x)~dx = \frac{1}{q^d} \sum_{x\in \mathbb F_q^d} f(x) \quad\mbox{for}~~f:(\mathbb F_q^d, dx) \to \mathbb C,$$
and
$$ \int_{\mathbb F_q^d} g(m)~dm = \sum_{m\in \mathbb F_q^d} g(x) \quad\mbox{for}~~g:(\mathbb F_q^d, dm) \to \mathbb C.$$
Let $\chi$ denote a fixed nontrivial additive character of $\mathbb F_q.$ Our results will be independent of our choice of character.\\

Given a function $g: (\mathbb F_q^d, dm) \to \mathbb C,$  the Fourier transform of $g$ is given by
$$ \widehat{g}(x)=\int_{\mathbb F_q^d} \chi(-x\cdot m) g(m)~dm :=\sum_{m\in \mathbb F_q^d}\chi(-x\cdot m) g(m)  \quad \mbox{for} ~~ x\in (\mathbb F_q^d, dx).$$
On the other hand,  if $f: (\mathbb F_q^d, dx) \to \mathbb C,$ then the inverse Fourier transform of $f$ is given by
$$ f^\vee(m)=\int_{\mathbb F_q^d} \chi(m\cdot x) f(x)~dx := \frac{1}{q^d} \sum_{x\in \mathbb F_q^d}\chi(m\cdot x) f(x)  \quad \mbox{for} ~~ m\in (\mathbb F_q^d, dm).$$
Recall that by orthogonality we have
$$ \sum_{x\in \mathbb F_q^d} \chi(m\cdot x)=\left\{ \begin{array}{ll} 0  \quad&\mbox{if}~~ m\neq (0,\dots,0)\\
q^d  \quad &\mbox{if} ~~m= (0,\dots,0), \end{array}\right.$$
where  $m\cdot x$ is the usual dot-product. By the above orthogonality relation of $\chi$, we obtain  Plancherel's theorem which states
$|| f^\vee\|_{L^2(\mathbb F_q^d, dm)} = \|f\|_{L^2(\mathbb F_q^d, dx)}.$ Namely, Plancherel's theorem yields
$$ \sum_{m\in \mathbb F_q^d} |f^\vee(m)|^2 = \frac{1}{q^d} \sum_{x\in \mathbb F_q^d} |f(x)|^2.$$
Given functions $f, h:(\mathbb F_q^d, dx) \to \mathbb C,$  the convolution $f\ast h$ is defined as
$$ f\ast h(y)=\int_{\mathbb F_q^d} f(y-x) h(x)~dx=\frac{1}{q^d} \sum_{x\in \mathbb F_q^d} f(y-x) h(x)\quad \mbox{for}~~y\in (\mathbb F_q^d, dx).$$
Notice that $(f\ast h)^\vee(m)=f^\vee(m) h^\vee(m)$ for $m\in (\mathbb F_q^d, dm).$\\

Given an algebraic variety $V \subset (\mathbb F_q^d, dx)$,  we endow $V$ with the normalized surface measure $\sigma$ which is defined by the relation 
$$ \int_V f(x)~d\sigma(x) := \frac{1}{|V|} \sum_{x\in V} f(x),$$
where $|V|$ denotes the cardinality of the set $V$ and $f: (\mathbb F_q^d, dx) \to \mathbb C.$
Notice that  $d\sigma(x)$ is replaced by $(q^d/|V|)~ V(x)~ dx.$
Throughout this paper, we shall identify the set $V\subset \mathbb F_q^d$ with the characteristic function $\chi_{V}$ on $V$, so that $V(x) = 1$ for $x \in V$ and $V(x) = 0$ otherwise.

\subsection{Definition of the restricted averaging operator}
With the above notation the averaging operator $A$ is defined by
$$ Af(y)=f\ast \sigma(y)=\int_V f(y-x) ~d\sigma(x) := \frac{1}{|V|} \sum_{x\in V} f(y-x),$$
where $f$ and $Af$ are defined on $(\mathbb F_q^d, dx).$ The averaging operator $A$ was initially introduced for finite fields by Carbery-Stone-Wright \cite{CSW}.  Sharp $L^p\to L^r$ estimates of the averaging operator $A$  were obtained in the case when $V$ is the sphere, the paraboloid, or the cone (\cite{ Ko13, KS11}).
Now, we consider a restricted operator $A_V$ defined by restricting $Af =f\ast \sigma$ to the algebraic variety $V.$
Namely, we have $A_Vf = Af|_V.$
We call the operator $A_V$  the restricted averaging  operator to the algebraic variety $V \subset (\mathbb F_q^d,dx).$
The main purpose of this paper is to study $L^p\to L^r$ estimates for the restricted averaging operator $A_V$.

\begin{problem}\label{1.1}(Restricted averaging problem) Let $\sigma$ be the normalized surface measure on the variety $V \subset (\mathbb F_q^d, dx).$
For $1\le p, r\le \infty,$  we define $A_V(p\to r)$ as the smallest constant $C>0$ such that the estimate
\begin{equation}\label{deRA}\|f\ast \sigma\|_{L^r(V, \sigma)} \le C \|f\|_{L^p(\mathbb F_q^d, dx)} \end{equation}
holds for all functions $f: \mathbb F_q^d\to \mathbb C.$
The quantity $A_V(p\to r)$ may depend on $q$, the cardinality of the underlying finite field $\mathbb F_q.$ The restricted averaging problem to $V$ is to determine all pairs $(p,r)$ such that $1\le p,r\le \infty$ and
$A_V(p\to r)$ is  independent of the field size $q.$
\end{problem}

 For positive numbers $A$  and $B,$ we use $A\lesssim B$ if there is a constant $C>0$ independent of the field size $q$ such that $A\le C B.$ We also use $A\sim B$ to indicate that $A\lesssim B$ and $B\lesssim A.$   In this setting, the restricted averaging problem is to find all pairs $(p, r)$  such that  $1\le p, r \le \infty$ and $A_V(p\to r) \lesssim 1.$ Here, we again stress that the implicit constant in $\lesssim$ is allowed to depend on $d, p,r$ but it must be independent of $q=|\mathbb F_q|.$
\begin{remark} \label{remark1.2} Since $\| 1\|_{L^s(V, \sigma)}=1=\|f\|_{L^s(\mathbb F_q^d, dx)}$ for all $1\le s\le\infty,$  we see from  H\"older's inequality that
$$ A_V(p_2\to r) \le A_V(p_1\to r) \quad \mbox{for}\quad 1\le \ p_1\le p_2 \le \infty$$
and
$$ A_V(p\to r_1)\le A_V(p\to r_2) \quad \mbox{for}\quad 1\le r_1\le r_2 \le \infty$$
which will allow us to reduce the analysis below to certain endpoint estimates.
\end{remark}
As usual, we denote by $A_V^*$ the adjoint operator of the restricted averaging operator to $V.$ Since
$$ \langle A_V f,~ h \rangle_{L^2(V, \sigma)} =\langle f,~ A_V^* h\rangle_{L^2(\mathbb F_q^d, dx)},$$
the adjoint operator $A_V^*$ of $A_V$ is given by
$$ A_V^* h (y) =  \frac{q^d}{|V|^2} \sum_{x\in V} V(x-y)~ h(x)$$
where $h: (V, \sigma) \to \mathbb C$ and $y\in (\mathbb F_q^d, dx).$ Observe that if $V=-V:=\{x\in \mathbb F_q^d: -x\in V\},$ then
$$A_V^* h =\frac{q^{2d}}{|V|^2} (hV)\ast V. $$

By duality, if $1<p,r<\infty,$ then  the estimate (\ref{deRA}) implies that
\begin{equation}\label{dual} \|A_V^* h\|_{L^{p^\prime}(\mathbb F_q^d, dx)} \le C  \|h\|_{L^{r^\prime}(V, \sigma)}~~\mbox{for all}~~ h: V\to \mathbb C,\end{equation}
where $p^\prime =p/(p-1)$ and $r^\prime =r/(r-1).$ We define $A_V^*(r'\to p')$ as the best constant $C>0$ such that  the estimate \eqref{dual} holds.
It follows that  $A_V(p\to r) \lesssim 1 \iff A_V^*(r'\to p')\lesssim 1.$

\subsection{Statement of main results}
Results on the restricted averaging problem are based on geometric properties of the underlying variety $V\subset \mathbb F_q^d.$
The structure of the variety $V$ can be explained in terms of the inverse Fourier transform of the normalized surface measure $\sigma$  on $V.$
Recall that the inverse Fourier transform $\sigma^\vee$ of the surface measure $\sigma$ is defined by
$$ \sigma^\vee(m)=\int_{V} \chi(m\cdot x)~d\sigma(x)=\frac{1}{|V|} \sum_{x\in V}\chi(m\cdot x)\quad \mbox{for}~~m\in (\mathbb F_q^d, dm).$$
 We shall derive certain results on varieties which possess general properties of hypersurfaces such as spheres, paraboloids, or cones in finite fields.  To precisely state  our main results, we need to classify
varieties $V$ according to their Fourier decay.

\begin{definition}  An algebraic variety $V\subset  (\mathbb F_q^d, dx)$ will be called a regular variety
if $|V|\sim q^{d-1}$ and $ |\sigma^\vee (m)|\lesssim q^{-(d-1)/2} $  for all $m\in \mathbb F_q^d\setminus (0,\ldots,0),$ where $\sigma$ denotes the normalized surface measure on the variety $V.$
\end{definition}

As the first main result, we obtain the sharp mapping properties of the restricted averaging operator to a regular variety (see Figure \ref{figure}).
\begin{theorem}\label{mainthm1} Let $\sigma$ be the normalized surface measure on  a regular variety $V\subset (\mathbb F_q^d, dx).$  Then we have
$A_V(p\to r) \lesssim 1$ if and only if $(1/p, 1/r)$ lies on the convex hull of  points $(0,0), (0, 1), ( (d-1)/d, 1),$ and $((d-1)/d, 1/d).$
\end{theorem}

It is well known (\cite{IR07, MT04}) that typical examples of regular varieties in $\mathbb F_q^d$ are the sphere $S_j:=\{ x\in \mathbb F_q^d: x_1^2+x_2^2+\cdots+x_d^2=j\ne 0\}$ and the paraboloid $P:=\{x\in \mathbb F_q^d: x_1^2+x_2^2+ \cdots+x_{d-1}^2 = x_d\}$. Thus,  Theorem \ref{mainthm1} provides us with the best result on the restricted averaging problem for the sphere $S_j$ and the paraboloid $P$.
However, if a variety $V\subset \mathbb F_q^d$ is not a regular variety, then it may not be simple to prove the sharp $L^p \to L^r$ estimates of the restricted averaging operator, because the  variety $V$ may contain a large dimensional affine subspace which no longer has any curvature.   
Recall that a cone $C$ in $\mathbb F_q^d$ is defined as
\begin{equation}\label{defcone} C=\{x\in \mathbb F_q^d: x_1^2+x_2^2 + \cdots+ x_{d-2}^2=x_{d-1}x_d\}.\end{equation}
The regular property of the cone $C\subset \mathbb F_q^d$  depends on the dimension $d\ge 3.$
In fact, we shall see from  Corollary \ref{cor4.4} in Section \ref{sec4} that  the cone $C \subset \mathbb F_q^d$ is not a regular variety in even dimensions $d\ge 4$  but  the cone $C$ is a regular variety in odd dimensions $d\ge 3$.
Hence,   if the dimension, $d\ge 3,$ is odd, then   the sharp $L^p \to L^r$ estimates for the  operator $A_C$ follows immediately from Theorem \ref{mainthm1}.
For this reason,  we shall focus on studying this problem for the cone in even dimensions $d\ge 4.$
In this paper, except for endpoints, we shall establish the sharp mapping properties of the restricted averaging operator to the cone $C$ in even dimensions.
More precisely we have the following result (see Figure \ref{figure}).
\begin{theorem}\label{mainthm2} Let $\sigma_c$ be the normalized surface measure on the cone $C \subset \mathbb F_q^d$ defined as in $\eqref{defcone}.$
Denote by $\Omega$ the convex hull of points $(0,0), (0,1), ((d-1)/d, 1),$
$$P_1:= \left(\frac{d-1}{d}, \frac{1}{d-2}  \right)~~\mbox{and}~~ P_2:=\left(\frac{d^2-3d+2}{d^2-2d+2}, \frac{d-2}{d^2-2d+2}  \right).$$
If the dimension, $d\geq 4,$ is even, then we have the following results:\\
\begin{enumerate}
\item
If $(1/p, 1/r) \in \Omega\setminus\{P_1,P_2\}, $ then $A_C(p\to r)\lesssim 1$

\item
 If $-1\in \mathbb F_q$ is a square number and $A_C(p\to r)\lesssim 1$, then $(1/p, 1/r) \in \Omega$.

\item
If we put $P_1=(1/p, 1/r),$ then the following restricted type inequality holds: $$ \|f\ast \sigma_c\|_{L^r(C, \sigma_c)}\lesssim \|f\|_{L^p(\mathbb F_q^d, dx)}\quad \mbox{for all characteristic functions} ~~f: \mathbb F_q^d \to \mathbb C.$$

\item
If we put $P_2=(1/p, 1/r)$, then  the weak-type estimate
$$ \|f\ast \sigma_c\|_{L^{r, \infty}(C, \sigma_c)} \lesssim \|f\|_{L^{p}(\mathbb F_q^d, dx)}$$
holds.

\end{enumerate}

\end{theorem}
\begin{figure}
\centering
\includegraphics[width=0.8\textwidth]{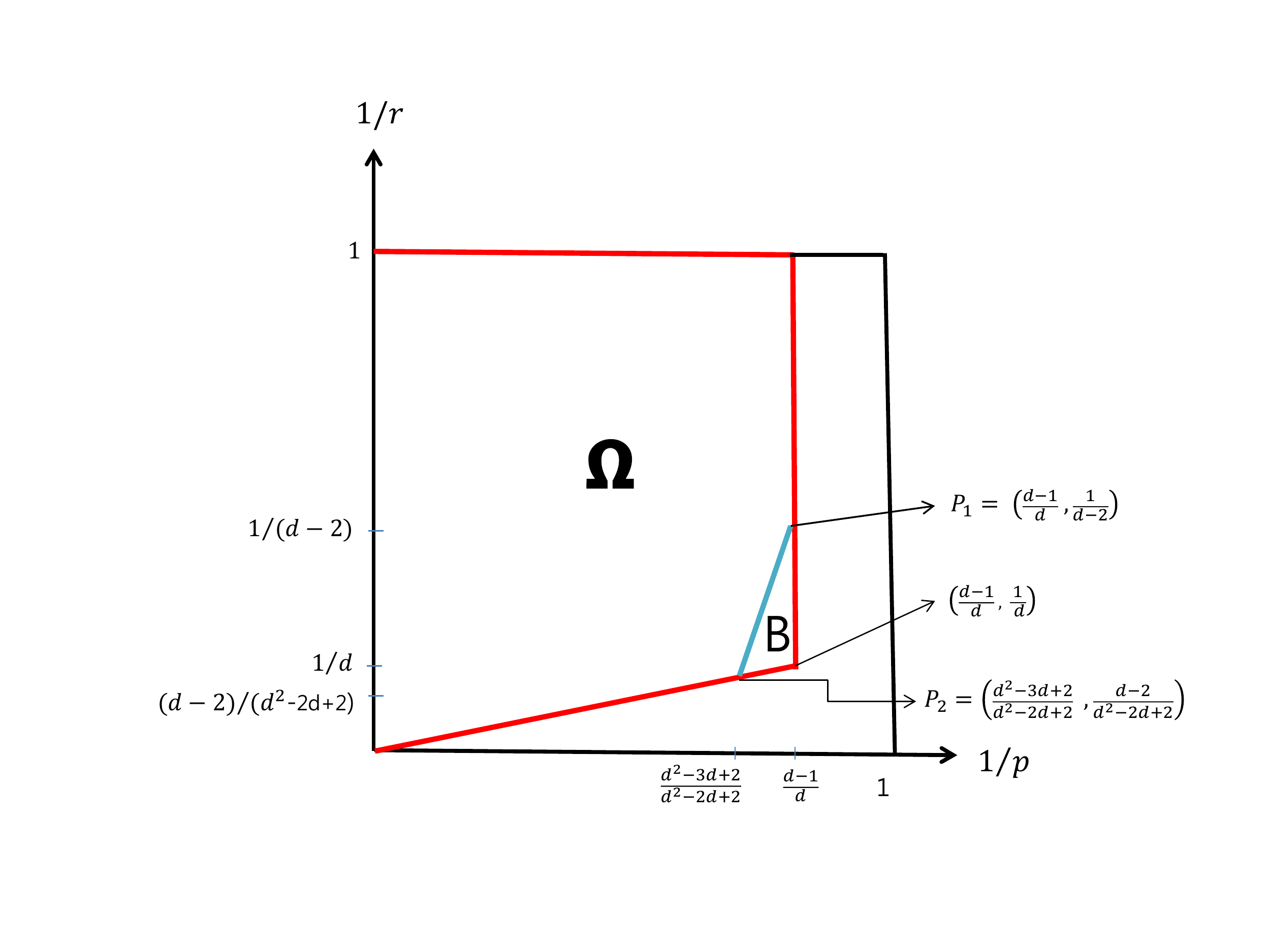}
\vspace{-1cm}
\caption{ The region $\Omega\cup B$ is related to Theorem \ref{mainthm1} which gives optimal  $L^p\to L^r$ result for regular varieties. On the other hand, the region $\Omega\setminus\{P_1, P_2\}$ indicates the conclusion of Theorem \ref{mainthm2} which provides us of the sharp results except for two endpoints $P_1, P_2$ for the cone in even dimensions $d\ge 4$ provided that $-1\in \mathbb F_q$ is a square number.}
\label{figure}
\end{figure}

\section{Necessary Conditions }\label{sec4}
 Let $V\subset \mathbb F_q^d$ be an algebraic variety with $|V|\sim q^{d-1}.$
We denote by $\sigma$ the normalized surface measure on $V.$
Then we have the following necessary conditions for the boundedness of $A_V(p\to r).$
\begin{lemma}\label{Lem2.1} Let $1\le p,r \le \infty.$
Assume that $A_V(p\to r)\lesssim 1.$ Then we must have
\begin{equation}\label{Ne12} \frac{1}{p} \le \frac{d-1}{d} \quad \mbox{and}\quad \frac{1}{p(d-1)} \le \frac{1}{r}.
\end{equation}
In addition, if we assume that $V$ contains an affine subspace $\Pi$ with $|\Pi|=q^\alpha$, then we must have
\begin{equation}\label{Affinene} \frac{d-\alpha}{p(d-1-\alpha)} \le \frac{1}{r} +1.
\end{equation}
\end{lemma}
\begin{proof}
For each ${\bf a} \in  \mathbb F_q^d$, define
$$ \delta_{\bf a}(x) =\left\{\begin{array}{ll}  1 \quad &\mbox{if} \quad x={\bf a} \\
                                         0 \quad &\mbox{if} \quad x\ne {\bf a} .\end{array}\right. $$
If we test \eqref{deRA} with $f$ equal to $\delta_{\bf 0},$ 
we obtain 
$$ \|\delta_{\bf 0}\ast \sigma\|_{L^r(V, \sigma)} \lesssim \|\delta_{\bf 0}\|_{L^p(\mathbb F_q^d, dx)}.$$
Since $\|\delta_{\bf 0}\|_{L^p(\mathbb F_q^d, dx)}=q^{-\frac{d}{p}}$  and $\|\delta_{\bf 0}\ast \sigma\|_{L^r(V, \sigma)}=1/|V|\sim q^{-(d-1)},$
we obtain a necessary condition $ 1/p \le (d-1)/d$ for the boundedness of  $A_V(p\to r).$
In order to obtain another necessary condition $ 1/(p(d-1))\le 1/r,$ we shall use the estimate $\eqref{dual}.$
If we test \eqref{dual} with $h$ equal to $\delta_{\bf a}$ for some ${\bf a}\in V$, we must have
$$  \|A_V^* \delta_{\bf a}\|_{L^{p^\prime}(\mathbb F_q^d, dx)} \lesssim  \|\delta_{\bf a}\|_{L^{r^\prime}(V, \sigma)}.$$
Notice that for $x\in (\mathbb F_q^d, dx),$  we have
$$A_V^* \delta_{\bf a} (x) =  \frac{q^d}{|V|^2} \sum_{y\in V} V(y-x)~ \delta_{\bf a} (y) =\frac{q^d}{|V|^2} V({\bf a}-x) . $$
It is not hard to see that
 $$\|A_V^* \delta_{\bf a}\|_{L^{p^\prime}(\mathbb F_q^d, dx)}=\left(\frac{|V|}{q^d}\right)^{1/p^\prime} \frac{q^d}{|V|^2} \sim q^{-d+2-1/p^\prime}$$
and
$$ \|\delta_{\bf a}\|_{L^{r^\prime}(V, \sigma)}=|V|^{-1/r^\prime} \sim q^{-(d-1)/r^\prime}.$$
Thus, we must have $-d+2-1/p^\prime\le -(d-1)/r^\prime$ which yields another necessary condition $ 1/(p(d-1))\le 1/r$ for the boundedness of  $A_V(p\to r).$
Now we prove the necessary condition $\eqref{Affinene}$.  Assume that the variety $V$ contains an affine subspace $\Pi$ with $|\Pi|=q^\alpha.$
If we test \eqref{dual} with $h$ equal to the characteristic function on $\Pi$, we obtain
\begin{equation}\label{eq1}\|A_V^* \Pi\|_{L^{p^\prime}(\mathbb F_q^d, dx)} \lesssim  \|\Pi\|_{L^{r^\prime}(V, \sigma)}.\end{equation}

It is clear that
\begin{equation}\label{eq2}
 \|\Pi\|_{L^{r^\prime}(V, \sigma)} =\left(\frac{|\Pi|}{|V|}\right)^{1/r^\prime} \sim q^{(\alpha-d+1)/r^\prime}.
\end{equation}
Since $\Pi \subset V$, we see that $A_V^* \Pi (x) =  \frac{q^d}{|V|^2} \sum_{y\in \Pi} V(y-x)$ for $x\in (\mathbb F_q^d, dx).$ It follows 
$$ \|A_V^* \Pi\|^{p^\prime}_{L^{p^\prime}(\mathbb F_q^d, dx)}=\frac{q^{dp^{\prime}-d}}{|V|^{2p^{\prime}}} \sum_{x\in \mathbb F_q^d} \left( \sum_{y\in \Pi} V(y-x)\right)^{p^\prime}.$$
Since $\Pi \subset V$ is an affine subspace,  we can choose a set $ M \subset \mathbb F_q^d$ such that $|M|=|\Pi|$ and $ \sum_{y\in \Pi} V(y-x) =|\Pi|$ for all $x\in M.$
Thus, we see that
$$  \|A_V^* \Pi\|^{p^\prime}_{L^{p^\prime}(\mathbb F_q^d, dx)}\ge \frac{q^{dp^{\prime}-d} |\Pi|^{1+p^{\prime}}}{|V|^{2p^{\prime}}} \sim \frac{q^{dp^{\prime}-d}q^{\alpha(1+p^{\prime})}}{q^{2p^\prime (d-1)}}=q^{(\alpha-d+2)p^{\prime} +\alpha-d},$$
where we used the conditions that $|\Pi|=q^\alpha$ and $|V|\sim q^{d-1}.$
It follows that
$$  \|A_V^* \Pi\|_{L^{p^\prime}(\mathbb F_q^d, dx)}\gtrsim q^{\alpha-d+2 +(\alpha-d)/p^{\prime}}.$$
From this inequality, $\eqref{eq1}$, and $\eqref{eq2}$, we must have
$$\alpha-d+2 +(\alpha-d)/p^{\prime}\le (\alpha-d+1)/r^\prime$$
which yields the necessary condition \eqref{Affinene} for the boundedness of  $A_V(p\to r).$
\end{proof}

\section{Proof of Theorem \ref{mainthm1}}

In this section, we shall prove Theorem \ref{mainthm1}. 
The necessary part for the boundedness of $A_V(p\to r)$ follows immediately from a direct consequence of \eqref{Ne12} in Lemma \ref{Lem2.1}.
It remains to prove the sufficient condition for the boundedness of $A_V(p\to r).$
To prove this,  observe  that  $A_V(\infty\to \infty) \lesssim 1.$ Now, By the Riesz-Thorin interpolation theorem (see Theorem 1.7 in \cite{BS88}) and  Remark \ref{remark1.2},  it will be enough to show
$$ A_V\left(\frac{d}{d-1}\to d\right) \lesssim 1.$$
Thus, our task is to establish the following estimate
\begin{equation}\label{keyestimate}
\|f\ast \sigma\|_{L^d(V, \sigma)} \lesssim \|f\|_{L^{d/(d-1)}(\mathbb F_q^d, dx)} ~~\mbox{for all}~~ f: (\mathbb F_q^d, dx) \to \mathbb C.
\end{equation}

For each $m\in (\mathbb F_q^d, dm)$, define $ K(m)= \sigma^\vee(m)-\delta_{\bf 0} (m).$
Then the measure $\sigma$ can be identified with the function $ \sigma(x)=\widehat{K}(x)+ \widehat{\delta_{\bf 0}}(x) =\widehat{K}(x) +1 $ for $x\in (\mathbb F_q^d, dx).$
To obtain the estimate $\eqref{keyestimate}$,  it suffices to prove the following estimates:
\begin{equation}\label{key1}
\|f\ast 1\|_{L^d(V, \sigma)} \lesssim \|f\|_{L^{d/(d-1)}(\mathbb F_q^d, dx)} ~~\mbox{for all}~~ f: (\mathbb F_q^d, dx) \to \mathbb C,
\end{equation}
\begin{equation}\label{key2}
\|f\ast \widehat{K}\|_{L^d(V, \sigma)} \lesssim \|f\|_{L^{d/(d-1)}(\mathbb F_q^d, dx)} ~~\mbox{for all}~~ f: (\mathbb F_q^d, dx) \to \mathbb C.
\end{equation}
Since $\max_{x\in V} |f\ast 1(x)| \le \|f\|_{L^1(\mathbb F_q^d, dx)},$  the estimate $\eqref{key1}$  follows by observing
$$ \|f\ast 1\|_{L^d(V, \sigma)} \le \|f\|_{L^1(\mathbb F_q^d, dx)} \|1\|_{L^d(V, \sigma)} \le \|f\|_{L^{d/(d-1)}(\mathbb F_q^d, dx)}.$$
Now notice that  $\eqref{key2}$ can be obtained by interpolating the following estimates:
\begin{equation} \label{kkey1}
\|f\ast \widehat{K}\|_{L^\infty(V, \sigma)} \lesssim q \|f\|_{L^1(\mathbb F_q^d, dx)} ~~\mbox{for all}~~ f: (\mathbb F_q^d, dx) \to \mathbb C.
\end{equation}
and
\begin{equation}\label{kkey2}
\|f\ast \widehat{K}\|_{L^2(V, \sigma)} \lesssim q^{\frac{-d+2}{2}}\|f\|_{L^{2}(\mathbb F_q^d, dx)} ~~\mbox{for all}~~ f: (\mathbb F_q^d, dx) \to \mathbb C.
\end{equation}
Thus, our task is to show that $\eqref{kkey1}$ and $\eqref{kkey2}$ hold.
Let us prove $\eqref{kkey1}.$
Observe
$$ \max_{y\in \mathbb F_q^d} |\widehat{K}(y)| =\max_{y\in \mathbb F_q^d} | \sigma(y)-1| = \max_{y\in \mathbb F_q^d} \left|\frac{q^dV(y)}{|V|} -1\right| \le \frac{q^d}{|V|} \sim q.$$
Then the estimate $\eqref{kkey1}$ follows by observing that for any $x\in V,$
$$  |f\ast \widehat{K}(x)| \le \left( \max_{y\in \mathbb F_q^d} |\widehat{K}(y)| \right) \frac{1}{q^d} \sum_{y\in \mathbb F_q^d} |f(x-y)| \lesssim q \|f\|_{L^1(\mathbb F_q^d, dx)}.$$
Finally, we shall prove the estimate $\eqref{kkey2}.$
From the definition of the function $K$ and the assumption on a regular variety $V$,  we see that
\begin{equation}\label{decaygood}
\max_{m\in \mathbb F_q^d}  |K(m)|  \lesssim q^{-\frac{(d-1)}{2}},
\end{equation}
which shall be used to prove $\eqref{kkey2}.$
In addition, we shall use the following restriction estimate.
\begin{lemma}\label{Lem3.1} Let $\sigma$ be the normalized surface measure on a variety $V \subset (\mathbb F_q^d, dx)$ with $|V|\sim q^{d-1}.$
Then we have
$$ \|\widehat{g}\|_{L^2(V, \sigma)} \lesssim q^{\frac{1}{2}} \|g\|_{L^2(\mathbb F_q^d, dm)}\quad \mbox{for all} ~~g:(\mathbb F_q^d, dm) \to \mathbb C.$$
\end{lemma}
\begin{proof} By duality, it suffices to prove the following extension estimate:
$$ \|(f\sigma)^{\vee} \|_{L^2(\mathbb F_q^d, dm)} \sim q^{\frac{1}{2}} \|f\|_{L^2(V, \sigma)} \quad \mbox{for all} ~~f: V \to \mathbb C.$$
Since $\sigma(x)=\frac{q^d}{|V|} V(x) $, it follows from Plancherel's theorem that
\begin{align*}\|(f\sigma)^{\vee} \|_{L^2(\mathbb F_q^d, dm)} &= \frac{q^d}{|V|}\| (fV)^\vee \|_{L^2(\mathbb F_q^d, dm)} = \frac{q^d}{|V|} \|fV\|_{L^2(\mathbb F_q^d, dx)} \\
&=\frac{q^{d/2}}{|V|^{1/2}} \|f\|_{L^2(V, \sigma)} \sim q^{\frac{1}{2}} \|f\|_{L^2(V, \sigma)}.\end{align*}
\end{proof}
To complete the proof of the estimate \eqref{kkey2}, we  write
$ \|f\ast \widehat{K}\|_{L^2(V, \sigma)} = \| \widehat{ f^\vee K}\|_{L^2(V, \sigma)},$ and apply Lemma \ref{Lem3.1} and  \eqref{decaygood}.
Then we see that
\begin{align*} \|f\ast \widehat{K}\|_{L^2(V, \sigma)} &= \| \widehat{ f^\vee K}\|_{L^2(V, \sigma)} \lesssim q^{1/2} \|f^\vee K\|_{L^2(\mathbb F_q^d, dm)} \\
&\lesssim q^{1/2} q^{-(d-1)/2}\|f^\vee \|_{L^2(\mathbb F_q^d, dm)} =q^{(-d+2)/2} \|f \|_{L^2(\mathbb F_q^d, dx)} ,\end{align*}
where  Plancherel's theorem was used to obtain the last equality. Thus, our proof is complete.

\section{Properties of the cone}
Recall that the cone $C \subset \mathbb F_q^d$ is defined as the set
$$ \{x\in \mathbb F_q^d: x_{d-1}x_d= x_1^2 + x_2^2 + \cdots + x_{d-2}^2\}$$
and  $\sigma_c$ denotes the normalized surface measure on the cone $C.$
In this section, we collect preliminary lemmas  which play an important role  in proving Theorem  \ref{mainthm2}.
Observing the decay of the inverse Fourier transform on the cone,  one may analyze the structural features of the cone.
As we shall see below, the inverse Fourier transform on the cone is closely related to the classical Gauss sum.
Recall that the Gauss sum $G$ is defined by
$$G=\sum_{t\in \mathbb F_q} \eta(t)\chi(t)~~\mbox{and}~~|G|=q^{\frac{1}{2}}$$
where $\eta$ denotes the quadratic character of $\mathbb F_q^*.$
Since $\sum_{t\in \mathbb F_q} \chi(at^2)=G\eta(a) $ for $a\in \mathbb F_q^*,$ Completing the square and using the change of variables, we see
\begin{equation}\label{square}
 \sum_{t\in \mathbb F_q} \chi(at^2+bt) =G \eta(t) \chi\left( \frac{b^2}{-4a}  \right) \quad\mbox{for}~~a\in \mathbb F_q^*, ~ b\in \mathbb F_q.
\end{equation}
The inverse Fourier transform on the cone  can be explicitly expressed.

\begin{lemma}\label{lem4.1} Let  $C\subset \mathbb F_q^d$ be the cone. For each $\ell\in \mathbb F_q$ and  $\xi=(\xi_1,\xi_2, \ldots, \xi_d)\in \mathbb F_q^d,$ define
$\Gamma_\ell(\xi)=\xi_1^2+\xi_2^2+\cdots+\xi_{d-2}^2-\ell\xi_{d-1}\xi_d.$ Then we have the following results:
\begin{enumerate}
\item
If the dimension, $d\ge 4,$ is even,   then
$$ C^\vee(m)= \left\{\begin{array}{ll} \frac {\delta_{{\bf 0}}(m)}{q}+\frac{(q-1)G^{d-2}}{q^d} ~~ &\mbox{for} ~~\Gamma_4(m)=0 \\
                                                                                     \frac{-G^{d-2}}{q^d}  ~~ &\mbox{for}~~ \Gamma_4(m)\ne 0. \end{array} \right.$$
\item
 If the dimension, $d\ge 3,$ is odd,   then
$$ C^\vee(m)= \left\{\begin{array}{ll} \frac {\delta_{{\bf 0}}(m)}{q}~~ &\mbox{for} ~~ \Gamma_4(m)=0\\
                                                           \frac{G^{d-1}}{q^d}~ \eta(-\Gamma_4(m)) ~~ &\mbox{for}~~ \Gamma_4(m)\ne 0. \end{array} \right.$$
\end{enumerate}
\end{lemma}
\begin{proof}  Notice that $C=\{x\in \mathbb F_q^d: \Gamma_1(x)=0\}.$
 By the definition of the inverse Fourier transform and the orthogonality relation of $\chi,$ we see that
\begin{align*} C^\vee(m)&=q^{-d} \sum_{x\in C} \chi(m\cdot x)=q^{-d}\sum_{x\in \mathbb F_q^d: \Gamma_1(x)=0} \chi(m\cdot x)\\
                                 & =q^{-d} \sum_{x\in \mathbb F_q^d}  \left( q^{-1} \sum_{s\in \mathbb F_q} \chi\left(s \Gamma_1(x)\right) \right) \chi(m\cdot x)\\
                                  &=\frac{\delta_{{\bf 0}}(m)}{q} + q^{-d-1} \sum_{s\ne 0} \sum_{x\in \mathbb F_q^d} \chi\left(s \Gamma_1(x)+\chi(m\cdot x)\right)\\
                                 &=\frac{\delta_{{\bf 0}}(m)}{q} + \frac{1}{q^{d+1}}\sum_{s\ne 0}  \sum_{x\in \mathbb F_q^d}\chi(s(x_1^2+\dots+x_{d-2}^2-x_{d-1}x_d)) \chi(m\cdot x)
\end{align*}
 By the formula \eqref{square} we see that
$$ C^{\vee}(m)= \frac{\delta_{{\bf 0}}(m)}{q}+ \frac{G^{d-2}}{q^{d+1}} \sum_{s\ne 0} \eta^{d-2}(s)\chi\left( \frac{m_1^2+\cdots + m_{d-2}^2}{-4s}\right)  I (m_{d-1}, m_d),$$
where we define

$$ I(m_{d-1}, m_d)= \sum_{x_{d-1}\in \mathbb F_q} \chi(m_{d-1}x_{d-1}) \sum_{x_d\in \mathbb F_q} \chi( (-sx_{d-1}+m_d) x_d ).$$
Compute the sum over $x_d\in \mathbb F_q$ by the orthogonality relation of $\chi$ and obtain that
$$ C^\vee(m)=\frac{\delta_{{\bf 0}}(m)}{q}+ \frac{G^{d-2}}{q^{d}} \sum_{s\ne 0} \eta^{d-2}(s) \chi\left(\frac {m_1^2+\cdots+m_{d-2}^2-4 m_{d-1}m_d}{-4s}\right).$$
Since $\eta^{d-2} =1$ for  even $d\geq 4$ ,  the first statement of Lemma \ref{lem4.1} follows. To prove the second part of Lemma \ref{lem4.1},
we first note that if the dimension $d\geq 3$, is odd, then $\eta^{d-2}(s)=\eta(s)=\eta(s^{-1}) $ for $s\neq 0.$ Therefore, if $m_1^2+\cdots+m_{d-2}^2-4m_{d-1}m_d=0$, the statement follows immediately from the orthogonality relation of $\eta.$  On the other hand, if $m_1^2+\cdots+m_{d-2}^2-4m_{d-1}m_d\neq 0,$ then the statement follows  from a change of variables, the definition of the Gauss sum, and properties of the quadratic character $\eta.$
\end{proof}

We need the following lemma.
\begin{lemma} \label{remark4.2} Let $C\subset \mathbb F_q^d$ be the cone. If $d\ge 3$, then we have
$|C|\sim q^{d-1}.$
\end{lemma}
\begin{proof}
From the definition of the inverse Fourier transform on the cone $C\subset \mathbb F_q^d$ and the conclusion of Lemma \ref{lem4.1},
 we  see  that
$$ C^\vee(0,\ldots,0)= \frac{|C|}{q^d} =\left\{ \begin{array}{ll} q^{-1} + \frac{(q-1)G^{d-2}}{q^d} \quad&\mbox{for even}~~d\ge 4\\
                                                             q^{-1} \quad&\mbox{for odd}~~d\ge 3.\end{array}\right.$$
Since the absolute value of the Gauss sum is $\sqrt{q}$  (namely, $|G|=\sqrt{q}$),  we conclude 
$$|C|\sim q^{d-1} \quad\mbox{for}~~d\ge 3.$$
\end{proof}

Since $|G|=\sqrt{q},$ the following result is immediate from  Lemma \ref{lem4.1}.

\begin{corollary}\label{cor4.3} Let $C\subset \mathbb F_q^d$ be the cone.  Assume that $m\in \mathbb F_q^d \setminus \{(0,\ldots, 0)\}.$
Then the following two statements hold:
\begin{enumerate}
\item
If the dimension, $d\ge 4,$ is even,   then
$$ |C^\vee(m)|\sim \left\{\begin{array}{ll} q^{-\frac{d}{2}} ~~ &\mbox{for} ~~\Gamma_4(m)=0 \\
                                                                                    q^{- \frac{(d+2)}{2}}  ~~ &\mbox{for}~~ \Gamma_4(m)\ne 0. \end{array} \right.$$
\item
 If the dimension, $d\ge 3,$ is odd,   then
$$ |C^\vee(m)|= \left\{\begin{array}{ll} 0~~ &\mbox{for} ~~ \Gamma_4(m)=0\\
                                                           q^{-\frac{(d+1)}{2}}~~ &\mbox{for}~~ \Gamma_4(m)\ne 0. \end{array} \right.$$
\end{enumerate}
\end{corollary}

The following result can be obtained by using Corollary \ref{cor4.3}.
\begin{corollary}\label{cor4.4} Let $\sigma_c$ be the normalized surface measure on the cone $C\subset \mathbb F_q^d.$
Suppose that $ (0,\ldots,0)\ne m \in \mathbb F_q^d.$ Then we have the following facts:
\begin{enumerate}
\item
If the dimension, $d\ge 4,$ is even,   then
$$ |\sigma_c^\vee(m)|\sim \left\{\begin{array}{ll} q^{-\frac{(d-2)}{2}} ~~ &\mbox{for} ~~\Gamma_4(m)=0 \\
                                                                                    q^{- \frac{d}{2}}  ~~ &\mbox{for}~~ \Gamma_4(m)\ne 0. \end{array} \right.$$
\item
 If the dimension, $d\ge 3,$ is odd,   then
$$ |\sigma_c^\vee(m)|= \left\{\begin{array}{ll} 0~~ &\mbox{for} ~~ \Gamma_4(m)=0\\
                                                           q^{-\frac{(d-1)}{2}}~~ &\mbox{for}~~ \Gamma_4(m)\ne 0. \end{array} \right.$$
\end{enumerate}
\end{corollary}
\begin{proof} Since $\sigma_c^\vee(m)=\frac{q^d}{|C|} C^\vee(m)$ and $|C|\sim q^{d-1}$ for $d\ge 3,$     the statement follows  immediately  from Corollary \ref{cor4.3}.
\end{proof}

We also need the following result.
\begin{lemma}\label{Lem4.5}Let $C^*=\{m\in \mathbb F_q^d:  \Gamma_4(m)=0\}.$ If the dimension, $d\ge 4$, is even, then we have
$$ \sum_{m\in C^*} |E^\vee(m)|^2 \lesssim q^{-d-1}|E|+ q^{-\frac{3d}{2}} |E|^2\quad\mbox{for all}~~E\subset (\mathbb F_q^d, dx).$$
\end{lemma}
\begin{proof} It is clear that $|C^*|\sim q^{d-1} $ for even $d\ge 4$, because $|C^*|=|C|\sim q^{d-1}.$
It follows that
\begin{align*}\sum_{m\in C^*} |E^\vee(m)|^2 &=q^{-2d} \sum_{x, y\in E} \sum_{m\in C^*} \chi(m\cdot (x-y)) \\
&=q^{-2d} \sum_{x,y\in E:x=y} |C^*| + q^{-2d} \sum_{x,y\in E: x\ne y} \sum_{m\in C^*} \chi(m\cdot (x-y))\\
&\lesssim q^{-2d} q^{d-1} |E| + q^{-2d} |E|^2  \max_{\beta\ne (0,\ldots,0)} \left| \sum_{m\in C^*} \chi(m\cdot \beta)\right|.
\end{align*}
Adapting the arguments used to prove the first part of Lemma \ref{lem4.1},  we see that
$$\max_{\beta\ne (0,\ldots,0)} \left| \sum_{m\in C^*} \chi(m\cdot \beta)\right| \lesssim q^{\frac{d}{2}}.$$
Combining this with the above estimate, we complete the proof.
\end{proof}
\section{Proof of Theorem \ref{mainthm2}}

To prove the statement $(2)$ of Theorem \ref{mainthm2}, it suffices by Lemma \ref{Lem2.1} to show that
if $-1$ is a square number and the dimension, $d\ge 4$, is even, then  the cone $C \subset \mathbb F_q^d$ contains a subspace $\Pi$ such that $|\Pi|=q^{d/2}.$ Now, define
$$ \Pi=\left\{(t_1, i t_1, \ldots, t_{(d-2)/2},  i t_{(d-2)/2}, s, 0) \in \mathbb F_q^d: s, t_j\in \mathbb F_q, ~j=1,2,\ldots, (d-2)/2\right\},$$
where $i$ denotes an element of $\mathbb F_q$ such that $i^2=-1.$ It is clear that $\Pi$ is a $d/2$-dimensional subspace contained in the cone $C.$ Thus, we complete  the proof of the statement $(2)$ of Theorem \ref{mainthm2}.\\

Next, notice that 
if the statements (3), (4) of Theorem \ref{mainthm2} are true, then the statement (1) of Theorem \ref{mainthm2} follows from  Remark \ref{remark1.2} and the Marcinkiewicz interpolation theorem (see Theorem 4.13 in \cite{BS88}).  
In conclusion, to complete the proof of Theorem \ref{mainthm2} it remains to show that the statements (3),(4) of Theorem \ref{mainthm2} hold, which shall be proved in the following subsections.

\subsection{Proof of the statement $(3)$ of Theorem \ref{mainthm2}}
We aim to prove that if the dimension, $d\ge 4,$ is even, then
$$\|E\ast \sigma_c\|_{L^{d-2}(C, \sigma_c)}\lesssim \|E\|_{L^{d/(d-1)}(\mathbb F_q^d, dx)}\quad\mbox{for all}~~E \subset (\mathbb F_q^d, dx).$$
As before, we can write $\sigma_c=\widehat{\rm{K}} +1$, where
we define $\rm{K}(m)= \sigma_c^\vee(m)-\delta_{\bf 0}(m) $ for $m\in (\mathbb F_q^d, dm).$
It suffices to prove that
\begin{equation}\label{ref5.1}
\|E\ast 1\|_{L^{d-2}(C, \sigma_c)}\lesssim \|E\|_{L^{d/(d-1)}(\mathbb F_q^d, dx)}\quad\mbox{for all}~~E \subset (\mathbb F_q^d, dx)
\end{equation}
and
\begin{equation}\label{ref5.2}
\|E\ast \widehat{\rm{K}}\|_{L^{d-2}(C, \sigma_c)}\lesssim \|E\|_{L^{d/(d-1)}(\mathbb F_q^d, dx)}\quad\mbox{for all}~~E \subset (\mathbb F_q^d, dx).\end{equation}

Notice that $\max_{x\in C} |E\ast 1(x)|\le \|E\|_{L^1(\mathbb F_q^d, dx)}.$ Then the inequality \eqref{ref5.1} follows because
$$\|E\ast 1\|_{L^{d-2}(C, \sigma_c)} \le \|E\|_{L^1(\mathbb F_q^d, dx)} \|1\|_{L^{d-2}(C, \sigma_c)}
= \|E\|_{L^1(\mathbb F_q^d, dx)} \le \|E\|_{L^{d/(d-1)}(\mathbb F_q^d, dx)}.$$
Notice that the estimate \eqref{ref5.2} can be obtained by interpolating the following two inequalities:
\begin{equation}\label{eq5.3}
\|E\ast \widehat{\rm{K}}\|_{L^{\infty}(C, \sigma_c)}\lesssim q\|E\|_{L^1(\mathbb F_q^d, dx)}\quad\mbox{for all}~~E \subset (\mathbb F_q^d, dx)
\end{equation}
and
\begin{equation}\label{eq5.4}
\|E\ast \widehat{\rm{K}}\|_{L^2(C, \sigma_c)}\lesssim q^{-\frac{(d-4)}{2}}\|E\|_{L^{2d/(d+2)}(\mathbb F_q^d, dx)}\quad\mbox{for all}~~E \subset (\mathbb F_q^d, dx)
\end{equation}

The inequality \eqref{eq5.3} follows from the arguments used to prove the estimate \eqref{kkey1}.
Now we prove the inequality \eqref{eq5.4}. By the property of convolution functions and Lemma \ref{Lem3.1},  we can write
$$ \|E\ast \widehat{\rm{K}}\|_{L^2(C, \sigma_c)}=\|\widehat{E^\vee \rm{K}}\|_{L^2(C, \sigma_c)} \lesssim q^{1/2} \|E^\vee {\rm{K}}\|_{L^2(\mathbb F_q^d, dm)}.$$
Thus, to prove \eqref{eq5.4} it will be enough to  show that
\begin{align}\label{eq5.5} \|E^\vee {\rm{K}}\|^2_{L^2(\mathbb F_q^d, dm)}
&\lesssim q^{-d+3} \|E\|^2_{L^{2d/(d+2)}(\mathbb F_q^d, dx)}\\ 
\nonumber &=\frac{|E|^{(d+2)/d}}{q^{2d-1}}\quad\mbox{for all}~~E \subset (\mathbb F_q^d, dx).\end{align}
Observe by the definition of $\rm{K}$ that $\rm{K}(0, \ldots, 0)=0$ and $\rm{K}(m)=\sigma^\vee_c(m) $ for $m\ne (0,\ldots,0).$  Let us prove the estimate \eqref{eq5.5}.\\

\noindent{\bf (Case I)} Assume that $|E|\ge q^{d/2}$ for even $d\ge 4.$  Applying $(1)$ of Corollary \ref{cor4.4} and Plancherel's theorem, we obtain
$$\|E^\vee {\rm{K}}\|^2_{L^2(\mathbb F_q^d, dm)}= \sum_{m\in \mathbb F_q^d} |E^\vee(m)|^2 |{\rm K}(m)|^2 \lesssim q^{-d+2} q^{-d} |E|=q^{-2d+2} |E|.$$
Since $ q^{-2d+2}|E| \le q^{-2d+1} |E|^{(d+2)/d} $ for $|E|\ge q^{d/2}$,   the estimate \eqref{eq5.5} holds.\\

\noindent{\bf (Case II)} Assume that $|E|\le q^{d/2}$ for even $d\ge 4.$ Using (1) of Corollary \ref{cor4.4} and Lemma \ref{Lem4.5}, we see that
\begin{align*}\|E^\vee {\rm{K}}\|^2_{L^2(\mathbb F_q^d, dm)}&= \sum_{m\in \mathbb F_q^d} |E^\vee(m)|^2 |{\rm{K}}(m)|^2 \\
                                                                                                     &=\sum_{\Gamma_4(m)=0} |E^\vee(m)|^2 |{\rm{K}}(m)|^2 + \sum_{\Gamma_4(m)\ne0} |E^\vee(m)|^2 |{\rm{K}}(m)|^2\\
                                                                                                     &\lesssim q^{-d+2} \left(q^{-d-1}|E|+ q^{-\frac{3d}{2}} |E|^2\right) + q^{-d} \sum_{m\in \mathbb F_q^d} |E^\vee(m)|^2\\
                                                                                                     &= q^{-2d+1}|E| + q^{(-5d+4)/2}|E|^2 + q^{-2d}|E|\\
                                                                                                     &\sim q^{-2d+1}|E| + q^{(-5d+4)/2}|E|^2.  \end{align*}
Since  the quantity in the last line is $\lesssim q^{-2d+1}|E|^{(d+2)/d}$ if $|E|\le q^{d/2}$,  the estimate \eqref{eq5.5} holds. We have completed the proof of the statement (3) of Theorem \ref{mainthm2}.
\subsection{Proof of the statement $(4)$ of Theorem \ref{mainthm2}}
By duality, it suffices to prove the following restricted type estimate:
$$ \|A_C^*F\|_{L^{\frac{d^2-2d+2}{d}}(\mathbb F_q^d, dx)} \lesssim \|F\|_{L^{\frac{d^2-2d+2}{d^2-3d+4}}(C, \sigma_c)}\quad \mbox{for all}~~F\subset (C, \sigma_c),$$
where we recall that
$$A_C^*F(x)=\frac{q^d}{|C|^2} \sum_{y\in C} C(y-x) F(y) .$$
Since $F\subset C$ and $C=-C$, we see that $A_C^*F=\frac{q^{2d}}{|C|^2} F\ast C.$
Hence, our task is to prove that
$$ \left\|\frac{q^{2d}}{|C|^2} F\ast C \right\|_{L^{\frac{d^2-2d+2}{d}}(\mathbb F_q^d, dx)} \lesssim \|F\|_{L^{\frac{d^2-2d+2}{d^2-3d+4}}(C, \sigma_c)}\quad \mbox{for all}~~F\subset (C, \sigma_c).$$
For each $m\in (\mathbb F_q^d, dm),$ define $H(m)=C^\vee(m) -\frac{|C|}{q^d} \delta_{{\bf 0}}(m).$ Then we can write $C(x)=\widehat{H}(x) + \frac{|C|}{q^d}$ for $x\in (\mathbb F_q^d, dx).$
To complete the proof, it is enough to show 
\begin{equation}\label{eq5.6}
\left\|\frac{q^{d}}{|C|} F\ast 1 \right\|_{L^{\frac{d^2-2d+2}{d}}(\mathbb F_q^d, dx)} \lesssim \|F\|_{L^{\frac{d^2-2d+2}{d^2-3d+4}}(C, \sigma_c)}\quad \mbox{for all}~~F\subset (C, \sigma_c)
\end{equation}
and
\begin{equation}\label{eq5.7}
\left\|\frac{q^{2d}}{|C|^2} F\ast \widehat{H} \right\|_{L^{\frac{d^2-2d+2}{d}}(\mathbb F_q^d, dx)} \lesssim \|F\|_{L^{\frac{d^2-2d+2}{d^2-3d+4}}(C, \sigma_c)}\quad \mbox{for all}~~F\subset (C, \sigma_c).
\end{equation}
Since $F\ast 1(x)=\frac{|F|}{q^d}$ for all $x\in (\mathbb F_q^d, dx),$  the inequality \eqref{eq5.6} follows by observing
$$ \left\|\frac{q^{d}}{|C|} F\ast 1 \right\|_{L^{\frac{d^2-2d+2}{d}}(\mathbb F_q^d, dx)}=~\frac{|F|}{|C|} \le \left(\frac{|F|}{|C|}\right)^{\frac{d^2-3d+4}{d^2-2d+2}} = \|F\|_{L^{\frac{d^2-2d+2}{d^2-3d+4}}(C, \sigma_c)}.$$
It remains to prove the inequality \eqref{eq5.7}.  To do this, we claim that the following two estimates hold:
\begin{equation}\label{eq5.8}
\left\|\frac{q^{2d}}{|C|^2} F\ast \widehat{H} \right\|_{L^{\infty}(\mathbb F_q^d, dx)} \lesssim q \|F\|_{L^1(C, \sigma_c)}\quad \mbox{for all}~~F\subset (C, \sigma_c)
\end{equation}
and
\begin{equation}\label{eq5.9}
\left\|\frac{q^{2d}}{|C|^2} F\ast \widehat{H} \right\|_{L^2(\mathbb F_q^d, dx)} \lesssim q^{\frac{-d^2+4d-2}{2d}} \|F\|_{L^{\frac{2d}{d+2}}(C, \sigma_c)}\quad \mbox{for all}~~F\subset (C, \sigma_c),
\end{equation}
which shall be proved below.
Notice that  the inequality \eqref{eq5.7} is obtained by interpolating \eqref{eq5.8} and \eqref{eq5.9}. 
To obtain the inequality \eqref{eq5.8},  we use Young's inequality and  observe that $\|\widehat{H}\|_{L^\infty(\mathbb F_q^d, dx)} \lesssim 1$. Then we see that
\begin{align*}\left\|\frac{q^{2d}}{|C|^2} F\ast \widehat{H} \right\|_{L^{\infty}(\mathbb F_q^d, dx)}&\le \frac{q^{2d}}{|C|^2}~ \|F\|_{L^1(\mathbb F_q^d, dx)}~ \|\widehat{H}\|_{L^\infty(\mathbb F_q^d, dx)}\\
&\lesssim \frac{q^d |F|}{|C|^2} =\frac{q^d}{|C|} ~\|F\|_{L^1(C, \sigma_c)}.\end{align*}
Since $|C|\sim q^{d-1}$,  the inequality \eqref{eq5.8} follows. Finally, we prove the inequality \eqref{eq5.9}. By Plancherel's theorem and the property of convolution functions, we can write
$$ \left\|\frac{q^{2d}}{|C|^2} F\ast \widehat{H} \right\|_{L^2(\mathbb F_q^d, dx)} \sim~ q^2 \|F^\vee H\|_{L^2(\mathbb F_q^d, dm)}. $$
Comparing this with the right-hand side of \eqref{eq5.9}, we see that it suffices to show 
\begin{align}\label{eq5.10} \|F^\vee H\|^2_{L^2(\mathbb F_q^d, dm)}&\lesssim q^{\frac{-d^2-2}{d}} \|F\|^2_{L^{\frac{2d}{d+2}}(C, \sigma_c)}\\
\nonumber &\sim \frac{|F|^{(d+2)/d}}{q^{2d+1}} \quad \mbox{for all}~~F \subset (C, \sigma_c).\end{align}
From the definition of $H$, we see that $H(0,\ldots,0)=0$ and $H(m)=C^\vee(m)$ for $m\ne (0,\ldots,0).$
Let us prove \eqref{eq5.10}. \\

\noindent{\bf (Case 1)} Assume that $|F|\ge q^{d/2}$ for even $d\ge 4.$
From (1) of  Corollary \ref{cor4.3} and Plancherel's theorem,  we see that
$$ \|F^\vee H\|^2_{L^2(\mathbb F_q^d, dm)} \lesssim q^{-d }\|F^\vee\|^2_{L^2(\mathbb F_q^d, dm)} =q^{-d} \|F\|^2_{L^2(\mathbb F_q^d, dx)}= q^{-2d} |F|.$$
Since $q^{-2d}|F| \le q^{-2d-1} |F|^{(d+2)/d}$ if $|F|\ge q^{d/2},$ the inequality \eqref{eq5.10} holds in this case.\\

\noindent{\bf (Case 2)} Assume that $|F|\le q^{d/2}$ for even $d\ge 4.$  From (1) of Corollary \ref{cor4.3} and Lemma \ref{Lem4.5}, it follows that
\begin{align*}
 \|F^\vee H\|^2_{L^2(\mathbb F_q^d, dm)}&= \sum_{m\in \mathbb F_q^d} |F^\vee(m)|^2 |H(m)|^2\\
                                                                    &\lesssim q^{-d} \sum_{ \Gamma_4(m)=0} |F^\vee(m)|^2  + q^{-d-2} \sum_{\Gamma_4(m)\ne 0} |F^\vee(m)|^2\\
                                                                    &\lesssim  q^{-d} \left( q^{-d-1}|F|+ q^{-\frac{3d}{2}} |F|^2  \right) + q^{-d-2} \sum_{m\in \mathbb F_q^d} |F^\vee(m)|^2\\
                                                                    &=\left(q^{-2d-1}|F| + q^{-5d/2} |F|^2\right) + q^{-2d-2} |F|\\
                                                                    & \sim q^{-2d-1}|F| + q^{-5d/2} |F|^2 .
\end{align*}
Since the last value is $\lesssim  q^{-2d-1} |F|^{(d+2)/d}$ if $|F|\le q^{d/2},$  the inequality \eqref{eq5.10} also holds.
Thus, we complete the proof of the statement (4) of Theorem \ref{mainthm2}.\\

{\bf Acknowledgement :} The authors would like to thank anonymous referees for their valuable comments which help to improve the manuscript.
We also wish to thank David Corvert for fixing grammatical errors  and clarifying ambiguous sentences in the previous version of this paper.

\end{document}